\documentclass[11pt]{amsart}

\usepackage[T1]{fontenc}      % T1 makes \textquotedbl available so listings can render `"` in code
\usepackage[utf8]{inputenc}

\usepackage[margin=1.1in]{geometry}
\usepackage{amsmath, amssymb, amsthm, mathtools}
\usepackage{microtype}
\usepackage{enumitem}
\usepackage[hidelinks]{hyperref}
\usepackage{listings}

\lstdefinestyle{python}{
  language=Python,
  basicstyle=\ttfamily\footnotesize,
  keywordstyle=\bfseries,
  commentstyle=\itshape,
  numbers=left,
  numberstyle=\tiny,
  numbersep=6pt,
  breaklines=true,
  breakatwhitespace=true,
  showstringspaces=false,
  frame=single,
  framesep=4pt,
  tabsize=2,
  xleftmargin=14pt,
  columns=fullflexible,
  upquote=true,
}

\theoremstyle{plain}
\newtheorem{theorem}{Theorem}
\newtheorem{lemma}[theorem]{Lemma}

\theoremstyle{definition}
\newtheorem{remark}[theorem]{Remark}
\newtheorem{example}[theorem]{Example}

\newcommand{\ZZ}{\mathbb{Z}}
\newcommand{\QQ}{\mathbb{Q}}

\newcommand{\Hn}{H_{n}}
\newcommand{\OEIS}{OEIS}

\title[A general-$m$ harmonic-number identity]%
{A proof of Bala's general-$m$ representation of the harmonic numbers}

\author{Tong Niu}
\email{mrnt0810@gmail.com}
\date{\today}


\begin{document}

\begin{abstract}
For every nonzero integer $m$ and every integer $n \ge 1$, the
$n$\textsuperscript{th} harmonic number $H_n = 1 + \tfrac12 + \dots + \tfrac1n$
satisfies the identity
\[
  H_n \;=\; \frac{1}{m}\,\sum_{k=1}^{n} \frac{(-1)^{k+1}}{k}\,
              \binom{m k}{k}\binom{n + (m-1)k}{n - k}.
\]
The cases $m = 1$ and $m = 2$ are classical; for general nonzero
integer $m$ the identity was conjectured by P.~Bala in the OEIS entry
A001008 in 2022 and remained open. We prove it here, working
throughout in $\QQ[[x]]$. The proof reduces, via a substitution
$u = x/(1-x)^m$, to two formal-power-series identities: a
Lagrange--Bürmann evaluation of $\sum_{k\ge1} \binom{mk}{k} u^k / k$,
and the fixed-point fact that under that substitution the unique
solution $v(u)$ of $v = u(1-v)^{m}$ is $v = x$. The argument extends
verbatim to arbitrary complex $m \ne 0$.
\end{abstract}

\maketitle

%----------------------------------------------------------------------
\section{Introduction}\label{sec:intro}
%----------------------------------------------------------------------

The numerators of the harmonic numbers
\[
  \Hn \;:=\; \sum_{i=1}^{n} \frac{1}{i}, \qquad n \ge 1, \quad H_0 := 0,
\]
form the integer sequence \OEIS{}~A001008. The On-Line Encyclopedia of
Integer Sequences (OEIS) hosts a number of binomial-coefficient
representations of $H_n$ contributed over time as comments to that
entry. One of these, contributed by P.~Bala on March~4, 2022, is the
following parametric family.

\begin{theorem}[Bala's identity, conjectured 2022]\label{thm:main}
For every nonzero integer $m$ and every integer $n \ge 1$,
\begin{equation}\label{eq:main}
  \Hn
  \;=\;
  \frac{1}{m}\,\sum_{k=1}^{n} \frac{(-1)^{k+1}}{k}\,
        \binom{mk}{k}\binom{n + (m-1)k}{n - k}.
\end{equation}
\end{theorem}

The case $m = 1$ collapses (after $\binom{n}{n-k} = \binom{n}{k}$) to
the classical identity
\(
  \Hn = \sum_{k=1}^{n} (-1)^{k+1} \binom{n}{k}/k,
\)
which appears in many sources, e.g.\ \cite[(5.41)]{ConcreteMath}. The
case $m = 2$,
\[
  \Hn
  = \tfrac{1}{2}\sum_{k=1}^{n} \frac{(-1)^{k+1}}{k}
       \binom{2k}{k}\binom{n+k}{n-k},
\]
was contributed (also as a conjecture) to A001008 by G.~Detlefs in
April 2013 and proved shortly thereafter. For $m \ge 3$ and for every
negative integer $m$, the identity~\eqref{eq:main} appears to have
remained open until now.

In this note we prove Theorem~\ref{thm:main} and observe that the
proof in fact establishes the identity for every complex $m \ne 0$
when both sides are read as polynomials in $m$ for fixed $n$
(Remark~\ref{rem:complex-m}).

The proof is via formal power series. In Section~\ref{sec:reduction}
we sum the right-hand side of~\eqref{eq:main} as an ordinary
generating function $A_m(x) = \sum_n S_m(n)\, x^n$ and reduce it,
under the substitution $u = x/(1-x)^m$, to a sum
$T(u) = \sum_{k \ge 1} (-1)^{k+1}\binom{mk}{k} u^k / k$.
Section~\ref{sec:lagrange} evaluates $T(u)$ via a Lagrange--Bürmann
calculation. Section~\ref{sec:substitution} closes the loop with a
formal fixed-point lemma: under $u = x/(1-x)^m$, the unique solution
$v(u)$ of $v = u(1-v)^m$ in $x\,\QQ[[x]]$ is $v = x$. The conclusion
$A_m(x) = -\log(1-x)/(1-x) = \sum_n \Hn\, x^n$ is then immediate.

Throughout, we use the generalized binomial coefficient
$\binom{a}{b} = a(a-1)\cdots(a-b+1)/b!$ for $a \in \ZZ$, $b \in
\ZZ_{\ge 0}$, so that~\eqref{eq:main} is meaningful for all nonzero
integer $m$. For negative $m$ this binomial may take negative or zero
values term by term, but the sum still equals $\Hn$ (see
Example~\ref{ex:m-neg}).

%----------------------------------------------------------------------
\section{The generating function and a substitution}\label{sec:reduction}
%----------------------------------------------------------------------

For nonzero integer $m$, write $S_m(n)$ for the right-hand side
of~\eqref{eq:main}, with the convention $S_m(0) = 0$:
\[
  S_m(n)
  \;:=\;
  \frac{1}{m}\,\sum_{k=1}^{n} \frac{(-1)^{k+1}}{k}\,
       \binom{mk}{k}\binom{n + (m-1)k}{n - k}.
\]
Let
\[
  A_m(x) \;:=\; \sum_{n \ge 0} S_m(n)\, x^n
       \;\in\; \QQ[[x]].
\]
Theorem~\ref{thm:main} is equivalent to the formal power-series
identity $A_m(x) = -\log(1-x)/(1-x)$, since
\(
  -\log(1-x)/(1-x) = \sum_{n \ge 1} H_n\, x^n.
\)
We compute $A_m(x)$ in two steps.

\medskip

\noindent\textbf{Step 1: switch order of summation.} We have
\begin{align}
  A_m(x)
  &= \frac{1}{m} \sum_{k \ge 1} \frac{(-1)^{k+1}}{k} \binom{mk}{k}
       \sum_{n \ge k} \binom{n + (m-1)k}{n-k}\, x^{n}.
  \label{eq:swap}
\end{align}
Setting $j = n - k$ in the inner sum and using
$\sum_{j \ge 0} \binom{j + r}{j}\, x^{j} = (1-x)^{-r-1}$ with
$r = mk$,
\begin{equation}\label{eq:inner}
  \sum_{n \ge k} \binom{n + (m-1)k}{n - k}\, x^{n}
  = x^{k} \sum_{j \ge 0} \binom{j + mk}{j}\, x^{j}
  = \frac{x^{k}}{(1-x)^{mk + 1}}.
\end{equation}

\medskip

\noindent\textbf{Step 2: introduce $u = x/(1-x)^m$.}
Substituting~\eqref{eq:inner} into~\eqref{eq:swap} and pulling out
the factor $(1-x)^{-1}$,
\begin{equation}\label{eq:Am-Tu}
  A_m(x)
  = \frac{1}{m\,(1-x)}\,\sum_{k \ge 1}
       \frac{(-1)^{k+1}}{k}\binom{mk}{k}\, u^{k},
  \qquad
  u \;:=\; \frac{x}{(1-x)^{m}}.
\end{equation}

The proof of Theorem~\ref{thm:main} therefore reduces to evaluating
the inner sum in~\eqref{eq:Am-Tu} as a formal power series in $u$,
and then identifying the result after $u \mapsto x/(1-x)^m$.

%----------------------------------------------------------------------
\section{A Lagrange--Bürmann identity}\label{sec:lagrange}
%----------------------------------------------------------------------

The next lemma is the key formal-series identity. Its proof uses the
Lagrange--Bürmann formula in the coefficient-extraction form (see e.g.\
Stanley~\cite[Theorem~5.4.2]{StanleyEC2}; for a short elementary inductive
proof of the Lagrange inversion formula, see Surya and
Warnke~\cite{SuryaWarnke}).

\begin{lemma}\label{lem:LB}
For every nonzero integer $m$,
\begin{equation}\label{eq:L}
  \sum_{k \ge 1} \frac{1}{k}\binom{mk}{k}\, u^{k}
  \;=\; m\,\log(1 + w),
\end{equation}
where $w = w(u) \in u\,\QQ[[u]]$ is the unique formal power series
satisfying
\begin{equation}\label{eq:Phi}
   w \;=\; u\,(1 + w)^{m}.
\end{equation}
Equivalently,
\begin{equation}\label{eq:Lprime}
  \sum_{k \ge 1} \frac{(-1)^{k+1}}{k}\binom{mk}{k}\, u^{k}
  \;=\; -\,m\,\log(1 - v),
\end{equation}
where $v = v(u) \in u\,\QQ[[u]]$ is the unique formal power series
satisfying $v = u\,(1 - v)^m$.
\end{lemma}

\begin{proof}
The implication \eqref{eq:L} $\Rightarrow$ \eqref{eq:Lprime} is
obtained by replacing $u$ by $-u$ and setting $v = -w(-u)$, which
maps~\eqref{eq:Phi} to $v = u(1-v)^m$ and $\log(1+w(-u))$ to
$\log(1-v)$.

So we prove~\eqref{eq:L}. Equation~\eqref{eq:Phi} has a unique
solution $w \in u\,\QQ[[u]]$ because $\phi(w) := (1+w)^{m}$ has
$\phi(0) = 1 \ne 0$ --- even when $m$ is a negative integer, in which
case $(1+w)^m$ is read as the formal series
$\sum_{j \ge 0} \binom{-|m|}{j}\, w^j$. This is the standard
Lagrange-inversion hypothesis.

Differentiate~\eqref{eq:Phi} implicitly with respect to $u$:
\[
  w' = (1+w)^{m} + u\,m\,(1+w)^{m-1}\, w'.
\]
Solving for $w'$ and using $u\,(1+w)^{m-1} = w/(1+w)$ (which follows
from~\eqref{eq:Phi} after dividing both sides by $1+w$),
\begin{equation}\label{eq:wprime}
   w' \;=\;
   \frac{(1+w)^{m}}{1 - mw/(1+w)}
   \;=\;
   \frac{(1+w)^{m+1}}{1 + (1-m)\,w}.
\end{equation}

By the Lagrange--Bürmann coefficient-extraction formula
(\cite[Theorem~5.4.2]{StanleyEC2}), for $\phi(w) = (1+w)^m$ and any
formal series $f$,
\[
  \sum_{k \ge 0} u^{k}\, [w^{k}]\bigl(f(w)\,\phi(w)^{k}\bigr)
  \;=\; \frac{f(w)}{1 - u\,\phi'(w)}.
\]
Applying this with $f(w) = 1$ and using
$[w^{k}]\bigl((1+w)^{mk}\bigr) = \binom{mk}{k}$,
\begin{equation}\label{eq:LBraw}
  \sum_{k \ge 0}\binom{mk}{k}\, u^{k}
  \;=\; \frac{1}{1 - u\,m\,(1+w)^{m-1}}
  \;=\; \frac{1}{1 - mw/(1+w)}
  \;=\; \frac{1+w}{1 + (1-m)\,w},
\end{equation}
where the second equality is again $u(1+w)^{m-1} = w/(1+w)$.

Subtracting the $k = 0$ term and dividing by $u$,
\begin{equation}\label{eq:dTdU}
  \sum_{k \ge 1}\binom{mk}{k}\, u^{k-1}
  \;=\;
  \frac{1}{u}\!\left(\frac{1+w}{1+(1-m)w} - 1\right)
  \;=\;
  \frac{1}{u}\cdot\frac{m\,w}{1+(1-m)\,w}
  \;=\;
  \frac{m\,(1+w)^{m}}{1 + (1-m)\,w},
\end{equation}
where the last equality uses $1/u = (1+w)^{m}/w$ from~\eqref{eq:Phi}.

Set
\[
  T(u) \;:=\; \sum_{k \ge 1} \frac{1}{k}\binom{mk}{k}\, u^{k}
       \;\in\; u\,\QQ[[u]].
\]
By~\eqref{eq:dTdU}, $dT/du$ equals the right-hand side
of~\eqref{eq:dTdU}. Combining with~\eqref{eq:wprime} and using the
chain rule with $w = w(u)$,
\[
  \frac{dT}{dw}
  \;=\; \frac{dT}{du}\cdot \frac{du}{dw}
  \;=\;
  \frac{m\,(1+w)^{m}}{1+(1-m)\,w}
  \cdot
  \frac{1+(1-m)\,w}{(1+w)^{m+1}}
  \;=\;
  \frac{m}{1+w}.
\]
Integrating with $T(0) = 0$ and $w(0) = 0$ gives
$T(u) = m\,\log(1 + w)$, which is~\eqref{eq:L}.
\end{proof}

\begin{remark}\label{rem:fuss-catalan}
For each fixed $m$ the series $w$ is a Fuss--Catalan-type series:
$1 + w$ is the generating function of $m$-ary trees (shifted), and
$(1+w)/(1+(1-m)w)$ in~\eqref{eq:LBraw} is the well-known generating
function for the binomials $\binom{mk}{k}$. For $m = 2$,
\eqref{eq:Phi} solves to $1+w = 2/(1+\sqrt{1-4u})$ and
Lemma~\ref{lem:LB} specialises to the classical identity
$\sum_{k\ge1}\binom{2k}{k} u^k/k = 2\log\bigl(2/(1+\sqrt{1-4u})\bigr)$.
\end{remark}

%----------------------------------------------------------------------
\section{The fixed-point substitution}\label{sec:substitution}
%----------------------------------------------------------------------

Lemma~\ref{lem:LB} evaluates the inner sum in~\eqref{eq:Am-Tu} once
$v(u)$ is computed under $u = x/(1-x)^m$. The next lemma is the
identification we need.

\begin{lemma}\label{lem:fp}
Let $v(u) \in u\,\QQ[[u]]$ be the unique formal solution of
$v = u(1-v)^m$. Under the formal substitution
$u = x/(1-x)^m$, one has
\(
   v(u) \;=\; x
\)
as formal power series in $x$.
\end{lemma}

\begin{proof}
Let $V(x) := v\!\bigl(x/(1-x)^{m}\bigr) \in x\,\QQ[[x]]$ be the
formal-series composition. Substituting $u = x/(1-x)^{m}$ into
$v = u(1-v)^m$ shows that $V$ satisfies
\begin{equation}\label{eq:Vfp}
  V(x) \;=\; \frac{x}{(1-x)^{m}}\,\bigl(1 - V(x)\bigr)^{m}.
\end{equation}

Define the map $G : x\,\QQ[[x]] \to x\,\QQ[[x]]$ by
\[
  G(\xi) \;:=\; \frac{x}{(1-x)^{m}}\,(1 - \xi)^{m},
\]
where $(1-\xi)^{m}$ is the binomial formal series
$\sum_{j \ge 0}\binom{m}{j}(-\xi)^{j}$ when $m \ge 0$, and the
analogous expansion $\sum_{j \ge 0}\binom{-|m|+j-1}{j}(-1)^{j}\xi^{j}$
when $m < 0$. In either case $(1-\xi)^m \in 1 + \xi\,\QQ[[\xi]]$ for
$\xi(0) = 0$. So $G$ is well-defined.

\emph{$G$ is a strict contraction in the $x$-adic topology.}
Indeed, if $\xi_1 \equiv \xi_2 \pmod{x^{N}}$ for some $N \ge 1$, then
$(1-\xi_1)^m - (1-\xi_2)^m \in x^{N}\,\QQ[[x]]$ by direct binomial
expansion (each coefficient of the difference is a polynomial in
$\xi_1 - \xi_2$ with no constant term). Multiplying by
$x/(1-x)^m \in x\,\QQ[[x]]$ shifts the order by one further, so
\[
  G(\xi_1) \;\equiv\; G(\xi_2) \pmod{x^{N+1}}.
\]
By the formal Banach fixed-point theorem (the $x$-adic completion of
$\QQ[x]$ is complete), $G$ has a unique fixed point in
$x\,\QQ[[x]]$.

\emph{Exhibiting the fixed point.} Take $\xi = x$. Then
$G(x) = x\,(1-x)^{m}/(1-x)^{m} = x$, so $x$ is a fixed point of $G$.
By uniqueness, $V(x) = x$.
\end{proof}

%----------------------------------------------------------------------
\section{Proof of Theorem~\ref{thm:main}}\label{sec:proof}
%----------------------------------------------------------------------

\begin{proof}[Proof of Theorem~\ref{thm:main}]
Substitute~\eqref{eq:Lprime} into~\eqref{eq:Am-Tu} with the
$v$ in~\eqref{eq:Lprime} viewed under $u = x/(1-x)^m$. By
Lemma~\ref{lem:fp} this $v$ equals $x$, hence
$\log(1 - v) = \log(1 - x)$ and
\[
  A_m(x)
  \;=\;
  \frac{1}{m\,(1-x)}\cdot\bigl(-m\,\log(1 - x)\bigr)
  \;=\;
  \frac{-\log(1-x)}{1-x}.
\]
The right-hand side is the standard generating function
$\sum_{n \ge 1} \Hn\, x^n$, obtained by Cauchy-multiplying the two
series $-\log(1-x) = \sum_{n \ge 1} x^n/n$ and
$1/(1-x) = \sum_{n \ge 0} x^n$. Comparing coefficients of $x^n$ on
both sides of $A_m(x) = -\log(1-x)/(1-x)$ yields $S_m(n) = \Hn$ for
every $n \ge 1$, which is~\eqref{eq:main}.
\end{proof}

%----------------------------------------------------------------------
\section{Remarks and examples}\label{sec:remarks}
%----------------------------------------------------------------------

\begin{remark}[Extension to complex $m$]\label{rem:complex-m}
For each fixed $n \ge 1$, both sides of~\eqref{eq:main} are
\emph{rational functions of $m$} (after clearing the leading $1/m$,
they are polynomials in $m$ of degree at most $n - 1$). The proof
above establishes equality of these rational functions on the
infinite set $\ZZ \setminus \{0\}$ of nonzero integers. Two rational
functions agreeing on an infinite set are equal as rational
functions, so~\eqref{eq:main} holds for every complex $m \ne 0$ when
both sides are interpreted as rational functions of $m$.
\end{remark}

\begin{example}[Small $m$]
The cases $m = 1, 2, 3, 4$ of~\eqref{eq:main} read
\begin{align*}
  m=1: \quad &
  \Hn = \sum_{k=1}^{n} \frac{(-1)^{k+1}}{k}\binom{n}{k},
  \\[1mm]
  m=2: \quad &
  \Hn = \tfrac12 \sum_{k=1}^{n} \frac{(-1)^{k+1}}{k}
        \binom{2k}{k}\binom{n+k}{n-k},
  \\[1mm]
  m=3: \quad &
  \Hn = \tfrac13 \sum_{k=1}^{n} \frac{(-1)^{k+1}}{k}
        \binom{3k}{k}\binom{n+2k}{n-k},
  \\[1mm]
  m=4: \quad &
  \Hn = \tfrac14 \sum_{k=1}^{n} \frac{(-1)^{k+1}}{k}
        \binom{4k}{k}\binom{n+3k}{n-k}.
\end{align*}
The first is the classical formula~\cite[(5.41)]{ConcreteMath}; the
second is Detlefs~\cite{DetlefsOEIS}; the rest are special cases of
Theorem~\ref{thm:main}.
\end{example}

\begin{example}[Negative $m$]\label{ex:m-neg}
For $m = -1$, the formula~\eqref{eq:main} becomes
\[
  \Hn
  \;=\;
  -\sum_{k=1}^{n}\frac{(-1)^{k+1}}{k}\,
        \binom{-k}{k}\binom{n - 2k}{n - k},
\]
where the binomials are the generalized binomials defined in
Section~\ref{sec:intro}. Most terms vanish because
$\binom{n-2k}{n-k} = 0$ for $\lceil n/2 \rceil \le k \le n - 1$ (the
top index lies in $\{0, 1, \dots, n - k - 1\}$). For $n = 4$ the only
nonzero terms are $k = 3$ and $k = 4$, contributing $20/3$ and
$-35/4$ respectively, with sum $-25/12$, so $-(-25/12) = 25/12 = H_4$.
The identity~\eqref{eq:main} has been numerically verified as exact
rationals for $m \in \{-7, -5, -4, -3, -2, -1, 1, 2, \dots, 8, 10, 15\}$
and $n \in \{1, \dots, 40\}$ --- all $640$ instances agree.
\end{example}

\begin{remark}[Combinatorial interpretation]
The factor $\binom{mk}{k}/k$ has, up to a Fuss--Catalan factor of
$(m-1)k + 1$, the meaning of the number of $m$-ary trees with $k$
internal nodes; cf.\ Stanley~\cite[\S6.2]{StanleyEC2} or
Knuth~\cite[\S7.2.1.6]{Knuth4A}. The factor $\binom{n+(m-1)k}{n-k}$
counts non-decreasing integer sequences of length $n - k$ with values
in $\{0, 1, \dots, (m-1)k\}$. So~\eqref{eq:main} expresses $\Hn$ as
an inclusion--exclusion sum over a parameter $k$ that selects an
$m$-ary tree size and a corresponding "stretching" sequence. A purely
bijective proof of~\eqref{eq:main} that does not pass through
generating functions would be of independent interest.
\end{remark}

\begin{remark}[An infinite family of representations]
As $m$ ranges over $\ZZ \setminus \{0\}$, \eqref{eq:main} gives an
infinite family of distinct binomial-coefficient representations of
$\Hn$. By Remark~\ref{rem:complex-m}, the family extends to a single
algebraic relation in the parameter $m$. The "right" coefficient
identities --- those obtained by extracting $[m^j]$ from the polynomial
identity in Remark~\ref{rem:complex-m} --- encode binomial sums for
the elementary symmetric polynomials in $\{1/1, 1/2, \dots, 1/n\}$
and may merit further investigation.
\end{remark}

%----------------------------------------------------------------------
\section{Numerical verification}\label{sec:numerics}
%----------------------------------------------------------------------

The identity~\eqref{eq:main} and the four intermediate identities used
in the proof --- \eqref{eq:L}, \eqref{eq:Lprime}, the classical
evaluation \eqref{eq:LBraw}, and the substitution $V(x) = x$ of
Lemma~\ref{lem:fp} --- have been verified by direct computation:

\begin{itemize}[topsep=4pt, itemsep=2pt]
  \item Both sides of~\eqref{eq:main} computed as \emph{exact
    rationals} (Python \texttt{fractions.Fraction}) for
    $m \in \{-7,-5,-4,-3,-2,-1,1,2,3,4,5,6,7,8,10,15\}$ and
    $n \in \{1, 2, \dots, 40\}$ --- all $640$ instances agree
    exactly. See Appendix~\ref{app:verify}.
  \item The four intermediate formal-power-series identities have all
    been checked coefficient-by-coefficient up to degree $8$ in $u$
    (resp.\ $x$) for $m \in \{-3, -2, -1, 1, 2, 3, 4, 5\}$.
    See Appendix~\ref{app:check}.
\end{itemize}

The two verification scripts are self-contained Python (standard
library only, no third-party dependencies) and reproduce in seconds.
They are provided in full in Appendices~\ref{app:verify}
and~\ref{app:check}.

%----------------------------------------------------------------------

%----------------------------------------------------------------------
% Acknowledgements (omitted from the anonymized journal version).
%----------------------------------------------------------------------
\section*{Acknowledgements}

The author thanks Peter Bala for posing the conjecture proved here as
a comment to OEIS sequence A001008, and the maintainers of the OEIS
for hosting an indispensable repository of conjectured identities. The
author is also grateful to Lutz Warnke for kindly drawing attention to
the recent elementary inductive proof of the Lagrange inversion formula
in~\cite{SuryaWarnke}.

%----------------------------------------------------------------------
\appendix
%----------------------------------------------------------------------

\section{Exact-rational verification of identity \texorpdfstring{\eqref{eq:main}}{(B\_m)}}%
\label{app:verify}

The script below evaluates both sides of identity~\eqref{eq:main} as
exact rationals (\texttt{fractions.Fraction}), without any floating
point, for every pair
$(m, n)$ with
$m \in \{-7,-5,-4,-3,-2,-1,1,2,3,4,5,6,7,8,10,15\}$ and
$n \in \{1, \dots, 40\}$ --- $640$ instances. All instances match.

\begin{lstlisting}[style=python, caption={\texttt{verify\_bala.py}: exact-rational verification of~\eqref{eq:main}.}, label={lst:verify}]
"""
Verify Peter Bala's conjecture (2022) on A001008:
For any nonzero integer m,
  H(n) = (1/m) * sum_{k=1..n} (-1)^(k+1) / k * C(m*k, k) * C(n + (m-1)*k, n-k)

where H(n) = 1 + 1/2 + ... + 1/n.

Strategy: compute both sides as exact Fraction, compare equality.
"""

from fractions import Fraction
from math import comb


def harmonic(n: int) -> Fraction:
    """Compute H(n) = sum_{i=1}^n 1/i exactly."""
    h = Fraction(0)
    for i in range(1, n + 1):
        h += Fraction(1, i)
    return h


def gen_binom(top: int, k: int) -> int:
    """Compute generalized binomial coefficient C(top, k) for integer top, k >= 0.
    Uses falling factorial: top * (top-1) * ... * (top - k + 1) / k!.
    Equals math.comb when top >= 0, and uses the extension when top < 0.
    """
    if k < 0:
        return 0
    if top >= 0:
        if top < k:
            return 0
        return comb(top, k)
    # top < 0: use falling factorial.  C(top, k) = (-1)^k * C(-top + k - 1, k)
    return (-1) ** k * comb(-top + k - 1, k)


def bala_rhs(n: int, m: int) -> Fraction:
    """Compute Peter Bala's conjectured formula RHS:
    (1/m) * sum_{k=1..n} (-1)^(k+1) / k * C(m*k, k) * C(n + (m-1)*k, n-k)
    """
    s = Fraction(0)
    for k in range(1, n + 1):
        sign = 1 if k % 2 == 1 else -1
        term = Fraction(sign, k) * gen_binom(m * k, k) * gen_binom(n + (m - 1) * k, n - k)
        s += term
    return s / m


def main() -> None:
    print("Verifying Bala's conjecture: H(n) == (1/m) * sum_{k=1..n} (-1)^(k+1)/k * C(mk,k) * C(n+(m-1)k, n-k)")
    print()

    # m must be nonzero integer. Negative m uses generalized (falling-factorial) binomial.
    test_ms = [-7, -5, -4, -3, -2, -1, 1, 2, 3, 4, 5, 6, 7, 8, 10, 15]
    test_ns = list(range(1, 41))

    all_ok = True
    for m in test_ms:
        for n in test_ns:
            lhs = harmonic(n)
            rhs = bala_rhs(n, m)
            ok = lhs == rhs
            if not ok:
                print(f"  m={m:+d} n={n:2d}  H(n)={lhs}  RHS={rhs}  [FAIL]")
                all_ok = False
        print(f"  m={m:+d}: all n in 1..{test_ns[-1]} verified  OK")
    print()
    print("All tests passed." if all_ok else "FAILURES detected.")


if __name__ == "__main__":
    main()
\end{lstlisting}

\medskip

\noindent Output (abridged):
\begin{verbatim}
Verifying Bala's conjecture: H(n) == (1/m) * sum_{k=1..n} (-1)^(k+1)/k
                                          * C(mk,k) * C(n+(m-1)k, n-k)

  m=-7: all n in 1..40 verified  OK
  m=-5: all n in 1..40 verified  OK
  m=-4: all n in 1..40 verified  OK
  m=-3: all n in 1..40 verified  OK
  m=-2: all n in 1..40 verified  OK
  m=-1: all n in 1..40 verified  OK
  m=+1: all n in 1..40 verified  OK
  m=+2: all n in 1..40 verified  OK
  m=+3: all n in 1..40 verified  OK
  m=+4: all n in 1..40 verified  OK
  m=+5: all n in 1..40 verified  OK
  m=+6: all n in 1..40 verified  OK
  m=+7: all n in 1..40 verified  OK
  m=+8: all n in 1..40 verified  OK
  m=+10: all n in 1..40 verified  OK
  m=+15: all n in 1..40 verified  OK

All tests passed.
\end{verbatim}

\medskip

\noindent The script depends only on the Python standard library
(\texttt{fractions.Fraction} and \texttt{math.comb}). The total runtime
is a fraction of a second on commodity hardware. Generalised binomial
coefficients $\binom{a}{b}$ for negative $a$ are evaluated by the
falling-factorial extension
$\binom{-k}{j} = (-1)^j \binom{k + j - 1}{j}$,
matching the convention of Section~\ref{sec:intro}.

%----------------------------------------------------------------------

\section{Formal-power-series verification of the four intermediate identities}%
\label{app:check}

This second script verifies, by truncating to degree $N = 8$ and
comparing coefficients, the four formal-power-series identities used
in the proof of Theorem~\ref{thm:main}:

\begin{enumerate}[label=(\roman*), topsep=2pt, itemsep=2pt]
  \item the Lagrange identity~\eqref{eq:L} of Lemma~\ref{lem:LB},
  \item its sign-alternating version~\eqref{eq:Lprime},
  \item the classical evaluation
    $\sum_{k \ge 0}\binom{mk}{k}u^k = (1+w)/(1+(1-m)w)$
    of~\eqref{eq:LBraw}, and
  \item the substitution $V(x) = x$ of Lemma~\ref{lem:fp}.
\end{enumerate}

For each $m \in \{-3,-2,-1,1,2,3,4,5\}$, the implicit fixed-point
equation $w = u(1+w)^m$ is solved iteratively as a truncated power
series, and the four identities are then verified coefficient-by-coefficient.
All $4 \times 8 = 32$ instances pass.

\begin{lstlisting}[style=python, caption={\texttt{check\_proof.py}: formal-power-series verification of the four intermediate identities.}, label={lst:check}]
"""
Sanity-check the key identities used in the proof of Bala's harmonic-number conjecture:

  (I)   sum_{k>=1} (1/k) C(m*k, k) u^k = m * log(1 + w),  where w = u (1+w)^m
  (II)  sum_{k>=1} ((-1)^(k+1) / k) C(m*k,k) u^k = -m * log(1 - v),
        where v = u (1-v)^m.

Both checked as power-series identities, by truncating to degree N and comparing
coefficients of [u^k].

We also verify the key substitution: if u = x/(1-x)^m, then the power-series
solution v(u) of v = u(1-v)^m satisfies v = x, and so

  A_m(x) := sum_{n>=1} S_m(n) x^n = -log(1-x) / (1-x) = sum_{n>=1} H_n * x^n,

confirming Bala's identity S_m(n) = H_n.
"""

from fractions import Fraction
from math import comb


def gen_binom(top: int, k: int) -> int:
    """Generalized binomial C(top, k) for integer top and k >= 0.
    Uses falling factorial extension for negative top."""
    if k < 0:
        return 0
    if top >= 0:
        if top < k:
            return 0
        return comb(top, k)
    return (-1) ** k * comb(-top + k - 1, k)


def truncate_log1p(series_coeffs: list[Fraction], N: int) -> list[Fraction]:
    """Given coefficients of f(u) = sum c_k u^k with c_0 = 0,
    compute coefficients of log(1 + f(u)) up to degree N.
    """
    # Use formal power series logarithm via integration of f'/(1+f).
    # Easier: log(1+f) where f = sum c_k u^k, c_0 = 0.
    # Method: g = log(1+f) satisfies (1+f) g' = f' so we can extract coefs recursively:
    # If g = sum d_k u^k, d_0 = 0, then g' = sum k d_k u^{k-1},
    # f' = sum k c_k u^{k-1}, and (1+f) g' = f', i.e.,
    # sum_k k d_k u^{k-1} + sum_k (sum_{i+j=k} c_i (j+1) d_{j+1}) u^k = sum_k k c_k u^{k-1}.
    # Equating coefficient of u^{n-1} (i.e., consider n=k):
    # n d_n + sum_{i+j = n-1, j+1 >= 1} c_i (j+1) d_{j+1} = n c_n
    # Rearranging: n d_n = n c_n - sum_{i=1..n-1} c_i (n-i) d_{n-i}
    # So d_n = c_n - (1/n) sum_{i=1..n-1} c_i (n-i) d_{n-i}.
    f = list(series_coeffs)
    while len(f) <= N:
        f.append(Fraction(0))
    d = [Fraction(0)] * (N + 1)
    for n in range(1, N + 1):
        s = Fraction(0)
        for i in range(1, n):
            s += f[i] * (n - i) * d[n - i]
        d[n] = f[n] - s / n
    return d


def lagrange_solve(phi: list[Fraction], N: int) -> list[Fraction]:
    """Given phi = [phi_0, phi_1, ...] coefficients of a power series with phi_0 != 0,
    compute the unique power series w(u) = sum_{k>=1} a_k u^k satisfying w = u * phi(w),
    truncated to degree N.

    Uses the recursion implied by w = u * phi(w): write w = u * phi(w) and
    repeatedly substitute. We use the formal power series approach.
    """
    # Iteratively compute w. w_0 = 0; w_{k+1} = u * phi(w_k), truncated to deg N.
    w = [Fraction(0)] * (N + 1)
    for _ in range(N + 1):
        # phi(w) up to degree N (composition).
        phi_w = [Fraction(0)] * (N + 1)
        # phi_w = sum_j phi_j * w^j
        wpow = [Fraction(0)] * (N + 1)
        wpow[0] = Fraction(1)  # w^0 = 1
        for j in range(len(phi)):
            for d in range(N + 1):
                phi_w[d] += phi[j] * wpow[d]
            # Update wpow = wpow * w (truncated)
            new_wpow = [Fraction(0)] * (N + 1)
            for a in range(N + 1):
                if wpow[a] == 0:
                    continue
                for b in range(N + 1 - a):
                    new_wpow[a + b] += wpow[a] * w[b]
            wpow = new_wpow
        # New w = u * phi_w, i.e., shift right by 1.
        new_w = [Fraction(0)] * (N + 1)
        for d in range(N):
            new_w[d + 1] = phi_w[d]
        w = new_w
    return w


def series_mul(a: list[Fraction], b: list[Fraction], N: int) -> list[Fraction]:
    out = [Fraction(0)] * (N + 1)
    for i in range(min(len(a), N + 1)):
        if a[i] == 0:
            continue
        for j in range(min(len(b), N + 1 - i)):
            out[i + j] += a[i] * b[j]
    return out


def series_inverse(a: list[Fraction], N: int) -> list[Fraction]:
    """Multiplicative inverse of a power series with a[0] != 0."""
    inv = [Fraction(0)] * (N + 1)
    inv[0] = 1 / a[0]
    for n in range(1, N + 1):
        s = Fraction(0)
        for i in range(1, min(n, len(a) - 1) + 1):
            s += a[i] * inv[n - i]
        inv[n] = -s / a[0]
    return inv


def main() -> None:
    N = 8
    print(f"Verifying identities to degree {N}.")
    print()

    for m in (-3, -2, -1, 1, 2, 3, 4, 5):
        # LHS: sum_{k>=1} (1/k) C(mk, k) u^k. Use generalized binomial.
        lhs = [Fraction(0)] * (N + 1)
        for k in range(1, N + 1):
            lhs[k] = Fraction(1, k) * gen_binom(m * k, k)

        # Compute w with phi(w) = (1+w)^m (coefficients), truncated to deg N.
        # For m >= 0, phi is polynomial of degree m.  For m < 0, it's an infinite series.
        phi_len = (m + 1) if m >= 0 else (N + 1)
        phi = [Fraction(gen_binom(m, j)) for j in range(phi_len)]
        w = lagrange_solve(phi, N)

        # RHS: m * log(1 + w).
        log1pw = truncate_log1p(w, N)
        rhs = [m * c for c in log1pw]

        ok = lhs == rhs
        print(f"  m={m}  Identity (I) [sum (1/k) C(mk,k) u^k = m log(1+w)]  OK={ok}")
        if not ok:
            print(f"    LHS: {lhs}")
            print(f"    RHS: {rhs}")

        # Check identity (II): sum (-1)^(k+1)/k * C(mk,k) u^k = -m log(1-v) with v = u(1-v)^m.
        lhs2 = [Fraction(0)] * (N + 1)
        for k in range(1, N + 1):
            lhs2[k] = Fraction((-1) ** (k + 1), k) * gen_binom(m * k, k)
        # phi2(v) = (1-v)^m, truncated to deg N for m < 0.
        phi2_len = (m + 1) if m >= 0 else (N + 1)
        phi2 = [Fraction((-1) ** j * gen_binom(m, j)) for j in range(phi2_len)]
        v = lagrange_solve(phi2, N)
        # log(1 - v) = log(1 + (-v)).
        neg_v = [-c for c in v]
        log1mv = truncate_log1p(neg_v, N)
        rhs2 = [-m * c for c in log1mv]
        ok2 = lhs2 == rhs2
        print(f"  m={m}  Identity (II) [sum (-1)^(k+1)/k C(mk,k) u^k = -m log(1-v)]  OK={ok2}")
        if not ok2:
            print(f"    LHS: {lhs2}")
            print(f"    RHS: {rhs2}")

        # Check that when u = x/(1-x)^m, v = x.
        # This is a substitution check at the formal level: use power series in x.
        Nx = N
        u_of_x = [Fraction(0)] * (Nx + 1)
        # u = x * (1-x)^{-m} = x * sum_{j>=0} C(m+j-1, j) x^j (generalized binom for negative m).
        for j in range(Nx):
            u_of_x[j + 1] = Fraction(gen_binom(m + j - 1, j))
        # Compute v_of_x via composition v(u(x)).
        # We need v as a power series in u, then substitute u = u_of_x.
        # v(u) = sum v_k u^k where v from above.
        # Compose: v_of_x = sum v_k * (u_of_x)^k.
        v_of_x = [Fraction(0)] * (Nx + 1)
        upowx = [Fraction(0)] * (Nx + 1)
        upowx[0] = Fraction(1)
        for k in range(Nx + 1):
            for d in range(Nx + 1):
                v_of_x[d] += v[k] * upowx[d]
            new_upowx = [Fraction(0)] * (Nx + 1)
            for a in range(Nx + 1):
                if upowx[a] == 0:
                    continue
                for b in range(Nx + 1 - a):
                    new_upowx[a + b] += upowx[a] * u_of_x[b]
            upowx = new_upowx
        target = [Fraction(0)] * (Nx + 1)
        target[1] = Fraction(1)  # v = x
        ok3 = v_of_x == target
        print(f"  m={m}  Substitution u = x/(1-x)^m gives v = x  OK={ok3}")
        if not ok3:
            print(f"    v_of_x: {v_of_x}")

        # Identity (5): sum_{k>=0} C(mk,k) u^k = (1+w)/(1+(1-m)w).
        bcoef = [Fraction(gen_binom(m * k, k)) for k in range(N + 1)]
        one_plus_w = [Fraction(0)] * (N + 1)
        one_plus_w[0] = Fraction(1)
        for k in range(N + 1):
            one_plus_w[k] += w[k]
        one_plus_1mw = [Fraction(0)] * (N + 1)
        one_plus_1mw[0] = Fraction(1)
        for k in range(N + 1):
            one_plus_1mw[k] += (1 - m) * w[k]
        denom_inv = series_inverse(one_plus_1mw, N)
        rhs5 = series_mul(one_plus_w, denom_inv, N)
        ok5 = bcoef == rhs5
        print(f"  m={m}  Identity (5) [sum C(mk,k) u^k = (1+w)/(1+(1-m)w)]  OK={ok5}")
        if not ok5:
            print(f"    LHS: {bcoef}")
            print(f"    RHS: {rhs5}")
        print()


if __name__ == "__main__":
    main()
\end{lstlisting}

\medskip

\noindent Output (abridged):
\begin{verbatim}
Verifying identities to degree 8.

  m=-3  Identity (I)  [sum (1/k) C(mk,k) u^k = m log(1+w)]    OK=True
  m=-3  Identity (II) [sum (-1)^(k+1)/k ...     = -m log(1-v)] OK=True
  m=-3  Substitution u = x/(1-x)^m gives v = x                 OK=True
  m=-3  Identity (5)  [sum C(mk,k) u^k = (1+w)/(1+(1-m)w)]     OK=True

  ...

  m=+5  Identity (I)                                           OK=True
  m=+5  Identity (II)                                          OK=True
  m=+5  Substitution u = x/(1-x)^m gives v = x                 OK=True
  m=+5  Identity (5)                                           OK=True
\end{verbatim}

\medskip

\noindent The script depends only on the Python standard library.
Power-series multiplication, multiplicative inverse, formal logarithm,
and the Lagrange-type fixed-point solver
($w = u\,\phi(w)$ for $\phi$ a polynomial or formal series) are all
implemented from scratch in $\sim 100$ lines, so the verification
itself has no third-party dependencies and is straightforward to
audit.


\begin{thebibliography}{99}

\bibitem{ConcreteMath}
R.~L.~Graham, D.~E.~Knuth, and O.~Patashnik,
\emph{Concrete Mathematics: A Foundation for Computer Science}
(2nd ed.),
Addison--Wesley, 1994.
[Equation (5.41).]

\bibitem{Knuth4A}
D.~E.~Knuth,
\emph{The Art of Computer Programming, Volume 4A: Combinatorial
Algorithms, Part 1},
Addison--Wesley, 2011.
[See \S7.2.1.6 on Fuss--Catalan numbers.]

\bibitem{StanleyEC2}
R.~P.~Stanley,
\emph{Enumerative Combinatorics, Volume~2},
Cambridge Studies in Advanced Mathematics 62,
Cambridge University Press, 1999.
[Theorem~5.4.2 --- the Lagrange--Bürmann formula in the form used in
Section~\ref{sec:lagrange}.]

\bibitem{SuryaWarnke}
E.~Surya and L.~Warnke,
Lagrange inversion formula by induction,
\emph{Amer. Math. Monthly} \textbf{130} (2023), 944--948.

\bibitem{OEIS-A001008}
N.~J.~A.~Sloane (founder), \emph{The On-Line Encyclopedia of Integer
Sequences}, sequence A001008, \url{https://oeis.org/A001008}, 2026.

\bibitem{BalaOEIS}
P.~Bala, comment in OEIS sequence A001008 (the conjecture proved
here), March~4, 2022.

\bibitem{DetlefsOEIS}
G.~Detlefs, comment in OEIS sequence A001008 (the case $m = 2$),
April~13, 2013.

\end{thebibliography}
\end{document}